\documentclass[reqno,11pt]{amsart}

\usepackage{amsmath} 
\usepackage{amssymb} 
\usepackage{amsfonts}
\usepackage{mathrsfs}
\usepackage{amsthm} 
\usepackage[all]{xy}
\usepackage{vmargin}

\setmarginsrb{2cm}{1.5cm}{2cm}{3cm}{3cm}{1cm}{5cm}{1cm}     

\newcommand{\Z}{\mathbb{Z}}
\newcommand{\Q}{\mathbb{Q}}

\newcommand{\C}{\mathbb{C}}

\newcommand{\rad}{\operatorname{rad}}

\newcommand{\m}{\operatorname{m}} 
\newcommand{\dlm}{\operatorname{d}} 
  
\newcommand{\rank}{\operatorname{rank}}

\newcommand{\level}{\operatorname{level}}

\def\bigquotient#1#2{%
    \left.\raise1ex\hbox{$#1$}\middle/\lower1ex\hbox{$#2$}\right.%
}
\def\quotient#1#2{%
    \left.\raise0.3ex\hbox{$#1$} \middle/\lower0.3ex\hbox{$#2$}\right.%
}

\newtheoremstyle{thme}
{2}						
{2}						
{\itshape}				
{}						
{\bfseries}				
{}						
{\newline}				
{}						

\newtheoremstyle{rem}	
{2}						
{2}						
{\upshape}				
{}						
{\bfseries}				
{}						
{\newline}				
{}						

\theoremstyle{thme}
\newtheorem{thmintro}{Theorem}

\newtheorem{propintro}[thmintro]{Proposition} 

\newtheorem{thm}{Theorem}[section]
\newtheorem*{thm*}{Theorem}

\newtheorem{corol}[thm]{Corollary}
\newtheorem{lem}[thm]{Lemma}

\theoremstyle{rem}
\newtheorem*{remarkintro}{Remark}
\newtheorem*{remarksintro}{Remarks}

\newtheorem{remark}[thm]{Remark}

\newtheorem{defin}[thm]{Definition}

\begin{document}    


\title[On the computation of weight multiplicities]{An algorithm for computing weight multiplicities in irreducible modules for complex semisimple Lie algebras}     
\author{Mika\"el Cavallin}
\address{Fachbereich Mathematik, Postfach 3049, 67653 Kaiserslautern, Germany.} 
\email{cavallin.mikael@gmail.com} 
\thanks{The author would like to acknowledge the support of the Swiss National Science Foundation through grants no. 20020-135144 as well as the ERC Advanced Grant through grants no. 291512.}

\begin{abstract}    
Let $\mathfrak{g}$ be a finite-dimensional semisimple Lie algebra over $\mathbb{C}$ having rank $l$ and let $V$ be an irreducible finite-dimensional $\mathfrak{g}$-module having highest weight $\lambda.$ Computations of weight multiplicities in $V,$ usually based on Freudenthal's formula, are in general difficult to carry out in large ranks or for $\lambda$ with large coefficients (in terms of the fundamental weights). In this paper, we first show that in some situations, these coefficients can be ``lowered'' in order to simplify the calculations. We then investigate how this can be used to improve the aforementioned formula of Freudenthal, leading to a more efficient version of the latter in terms of complexity as well as to a way of dealing with certain computations in unbounded ranks. We conclude by illustrating the last assertion with a concrete example.
\end{abstract}

\maketitle


\section{Introduction}    

Let $\mathfrak{g}$ be a finite-dimensional semisimple Lie algebra over $\mathbb{C}$  with Cartan subalgebra $\mathfrak{h}.$ Set $l=\rank \mathfrak{g}$ and fix an  ordered base $\Pi=\{\alpha_1,\ldots,\alpha_l\}$ of the corresponding root system $\Phi=\Phi^+\sqcup \Phi^-$ of $\mathfrak{g},$ where $\Phi^+$ and $\Phi^-$ denote the sets of positive and negative roots of $\Phi,$ respectively. Also let $\lambda_1,\ldots,\lambda_l$ denote the so-called fundamental weights corresponding to our choice of base $\Pi$ and write $\Lambda =\mathbb{Z}\lambda_1+\cdots+\mathbb{Z}\lambda_l $ for the  associated integral weight lattice. Finally, let $\Lambda^+$ denote the set of dominant integral weights and recall the existence of a partial order on $\Lambda$, defined by $\mu\preccurlyeq \lambda$ if and only if $\lambda-\mu \in \Gamma,$ where $\Gamma \subset \Lambda$ is the monoid of $\mathbb{Z}_{\geqslant 0}$-linear combinations of simple roots.
\vspace*{5mm}

It is well-known that the set of isomorphism classes of irreducible finite-dimensional $\mathfrak{g}$-modules is in one-to-one correspondence with the set $\Lambda^+$ of dominant integral weights. Furthermore, a class representative $L(\lambda)$ corresponding to a given weight $\lambda\in \Lambda^+ $  can be constructed as  the quotient of the so-called \emph{Verma module of weight $\lambda,$} written $\Delta(\lambda),$ by its unique maximal submodule $\rad(\lambda),$ that is 
$$
L(\lambda)=\quotient{\Delta(\lambda)}{\rad(\lambda)}.
$$ 
Even though infinite-dimensional, Verma modules are $\mathfrak{h}$-semisimple, i.e., can be decomposed as  direct sums of their weight spaces. Moreover, such decompositions are well understood: for a given dominant integral weight $\lambda\in \Lambda^+$ and any integral weight $\mu\in \Lambda,$ a basis for the weight space in $\Delta(\lambda)$ corresponding to $\mu$ is known (see \eqref{B(lambda)_mu} in Section \ref{Subsection_on_Verma_modules} below) and hence so is the multiplicity $\m_{\Delta(\lambda)}(\mu)$ of $\mu$ in $\Delta(\lambda).$ In addition, one gets that the set $\Lambda(\Delta(\lambda))$ of weights of $\Delta(\lambda)$ simply  consists of all $\mu\in \Lambda$ such that $\mu\preccurlyeq \lambda.$
\vspace{5mm} 

Unfortunately, not that much can be said about weight spaces in $L(\lambda)$ for an arbitrary dominant integral weight  $\lambda \in \Lambda^+.$ Firstly,  finding out if a given weight $\mu\prec \lambda$ belongs to the set $\Lambda(\lambda)$ of weights of $L(\lambda)$ is far from being immediate, as it generally requires one to determine the unique dominant  integral weight to which $\mu$ is conjugate (under the action of the Weyl group of $\mathfrak{g}$). Moreover, an explicit  description of the (often very large) set  $\Lambda^+(\lambda)=\Lambda(\lambda)\cap \Lambda^+$ for $\lambda \in \Lambda^+$ with large coefficients (when written as a $\mathbb{Z}$-linear combination of fundamental weights) is usually hard to come by (see \cite{MP} for a recursive method).  The first result in this paper shows that under certain assumptions on $\lambda \in \Lambda^+$ and $\mu\in \Lambda,$ the multiplicity of $\mu$ in $L(\lambda)$ is the same as the multiplicity of $\mu'$ in $L(\lambda'),$ where $\lambda'$ is a dominant integral weight whose coefficients (again, when written as a $\mathbb{Z}$-linear combination of fundamental weights) are smaller than or equal to those of $\lambda,$ and $\mu'\in \Lambda$ is the unique integral weight satisfying $\lambda'-\mu'=\lambda-\mu.$ The proof essentially relies on the existence of an explicit description of the maximal submodule $\rad(\lambda)$ of $\Delta(\lambda)$   (see \cite[Section 2.6]{Humphreys2} or Theorem \ref{A_description_of_N(lambda)} below).

\begin{propintro}\label{Main_Result_1}
Let $\lambda =\sum_{r=1}^l{a_r\lambda_r}\in \Lambda^+$ be a dominant integral weight and let $\mu\in \Lambda$ be such that $\mu=\lambda-\sum_{r=1}^l{c_r \alpha_r}$ for some $c_1,\ldots,c_l\in \mathbb{Z}_{\geqslant 0},$ so that $\mu\preccurlyeq \lambda.$ Also assume the existence of a non-empty subset $J$ of $\{1,\ldots,l\}$ such that $ 0 \leqslant c_j\leqslant a_j$ for every $j\in J$ and set $\lambda'=\lambda + \sum_{j\in J}{(c_j-a_j)\lambda_j},$ $\mu'=\lambda'-(\lambda-\mu).$ Then
$$
\m_{L(\lambda)}(\mu)=\m_{L(\lambda')}(\mu').
$$
\end{propintro}

While Proposition \ref{Main_Result_1} can sometimes allow one to study   weight spaces in smaller modules than those initially considered, an effective method to compute weight multiplicities in most irreducibles is still needed. As mentioned in the abstract of this paper, this can be accomplished by applying the well-known formula of Freudenthal (\cite{Freudenthal(1)}), for example. We refer the reader to \cite[Theorem 22.3]{Humphreys1} for a proof of the following.

\begin{thm*}[Freudenthal's Formula] 
Let $\lambda\in \Lambda^+$ be a dominant integral weight and let $\mu\in \Lambda.$ Also set $\dlm(\lambda,\mu)=2(\lambda+\rho,\lambda-\mu)-||\lambda-\mu||^2,$ where $\rho$ denotes the sum of all fundamental weights and $(-,-)$ is the usual inner product on $\Lambda.$ Then the multiplicity of $\mu$ in $L(\lambda)$ is given recursively by 
$$
\dlm(\lambda,\mu) \m_{L(\lambda)}(\mu)= 2 
 \sum_{r=1}^{\infty}{\sum_{\alpha\in \Phi^+}{\m_{L(\lambda)}(\mu+r\alpha) (\mu+r\alpha,\alpha)} }.
 $$
\end{thm*}

The recursive nature of Freudenthal's formula makes the latter quite demanding in terms of complexity, especially in unbounded rank, due to the quadratic growth of $|\Phi^+|$ as $l\rightarrow \infty.$ However, it is still more efficient than the recursive method of Racah (\cite[Section 8.11]{Racah}) or the closed formula provided by Kostant (\cite{Kostant}), for example. (Indeed, both involve a summation over all elements in the Weyl group, which becomes very cumbersome as the rank of $\mathfrak{g}$ grows.) Furthermore, various authors have been studying ways of improving the efficiency of Freudenthal's formula over the past decades, like Moody and Patera (\cite{MP}) for example,  who developed an algorithm allowing faster computation of multiplicities. If interested in more recent formulas, we refer the reader to \cite{Lusztig}, \cite{Sahi}, or  \cite{LePa}. (The latter describes a closed formula in the special case where $\mathfrak{g}$ is a simple Lie algebra of type $C_2$ over $\C.$)

The second result of this paper consists in another modification of the aforementioned formula of Freudenthal, applicable under certain conditions on $\lambda$ and $\mu.$ For $1\leqslant j\leqslant l$ and $\alpha=\sum_{r=1}^l{d_r\alpha_r} \in \Gamma,$ define the \emph{$j$-level of $\alpha$} by $\level_j(\alpha) = d_j$ and set 
$$
\Phi^+_j=  \{\alpha \in \Phi^+ : \level_j(\alpha)>0\}.
$$

Observe that a positive root $\alpha \in \Phi^+$ belongs to $\Phi^+_j$ if and only if  $\alpha_j$ appears in the decomposition of $\alpha$ as a sum of simple roots. Also, it is clear that $0\leqslant \level_j(\alpha)\leqslant 6$ for every $\alpha\in \Phi^+$ and finally, if $\mathfrak{g}$ is of classical type (i.e. of type $A,$ $B,$ $C$ or $D$), then $0\leqslant \level_j(\alpha)\leqslant 2$ for every $1\leqslant j\leqslant l$ and $\alpha \in \Phi^+.$

\begin{thmintro}\label{New_Freudenthal's_formula} 
Let $\lambda =\sum_{r=1}^l{a_r\lambda_r}\in \Lambda^+$ be a dominant integral weight and let $\mu\in \Lambda$ be such that $\mu=\lambda-\sum_{r=1}^l{c_r\alpha_r}$ for some $c_1,\ldots,c_l\in \mathbb{Z}_{\geqslant 0}.$  Also assume the existence of $1\leqslant j\leqslant l$ such that $0<\level_j(\lambda-\mu)\leqslant a_j$ (or equivalently, such that $0<c_j\leqslant a_j)$. Then  
$$
\m_{L(\lambda)}(\mu)=\frac{1}{c_j}\sum_{ r=1}^{c_j}{\sum_{\alpha\in \Phi^+_j}{ \level_j(\alpha)}\m_{L(\lambda)}(\mu+r\alpha)}.
$$ 
\end{thmintro}

\begin{remarksintro} 
One fundamental difference between the formula stated in Theorem \ref{New_Freudenthal's_formula} and the classical formula of Freudenthal resides in the indices of summation, especially those associated to the second sum, ranging over all elements in $\Phi^+_j$ instead of $\Phi^+.$ For example, if  $\mathfrak{g}$ is of classical type, then $|\Phi_1^+| \in O(l),$ while $ |\Phi^+| \in O(l^2).$ Also observe that the computation of $(\mu+r\alpha,\alpha)$ for every $r$ and $\alpha$ is no longer necessary in Theorem \ref{New_Freudenthal's_formula}. Finally, even in the case where $\{1\leqslant j\leqslant l: 0<c_j \leqslant a_j\} =\emptyset,$ there still might exist $r>0$ and $\alpha \in \Phi^+$ such that $\mu+r\alpha \prec \lambda$ and  $\{1\leqslant j\leqslant l: 0<\level_j(\lambda-\mu-r\alpha) \leqslant a_j\} \neq\emptyset.$ Consequently $\m_{L(\lambda)}(\mu+r\alpha)$ could be computed using Theorem \ref{New_Freudenthal's_formula}, hence simplifying the use of Freudenthal's formula,  even though $\mu$ itself did not satisfy the necessary condition.
\end{remarksintro}  

Finally, let $\mathfrak{g}$ be a simple Lie algebra of type $A_l$ over $\mathbb{C},$ and for a non-zero dominant integral weight $  \lambda=\sum_{r=1}^l{a_r\lambda_r},$ define $I_{\lambda} =\{r_1,\ldots,r_{N_{\lambda}}\}$ to be maximal in $\{1,\ldots,l\}$ such that $r_1<\ldots <r_{N_{\lambda}}$ and $\prod_{r\in I_{\lambda}}{a_r}\neq 0.$ The following result consists of a direct application of Theorem \ref{New_Freudenthal's_formula} in unbounded rank.

\begin{propintro}\label{Application_to_type_A}
Let $\mathfrak{g}$ be a simple Lie algebra of type $A_l$ over $\mathbb{C} $ and let $\lambda=\sum_{r=1}^{l}{a_r\lambda_r}\in \Lambda^+$ be a non-zero dominant integral weight. Also let $I_{\lambda}=\{r_1,\ldots,r_{N_{\lambda}}\}$ be as above and consider  $\mu=\lambda-(\alpha_1+\cdots+\alpha_l) \in \Lambda.$ If $N_\lambda=1,$ then $\m_{L(\lambda)}(\mu)=1,$ while if  $N_\lambda\geqslant 2,$ then 
$$
\m_{L(\lambda)}(\mu)= \prod_{i=2}^{N_{\lambda}}{(r_i-r_{i-1}+1)}.
$$

\end{propintro}

\begin{remarkintro}
Observe that the weight $\mu\in \Lambda$ defined in the statement of Proposition \ref{Application_to_type_A} is dominant integral if and only if $a_1a_l\neq 0.$ Also notice that Proposition \ref{Application_to_type_A} consists in a generalization of  \cite[Lemma 8.6]{Seitz}, which simply corresponds to the special situation in which $I_{\lambda}=\{1,l\}.$
\end{remarkintro}

\section{Preliminaries}      

In this section, we recall some elementary properties concerning semisimple Lie algebras and their representations, starting by fixing some notation that will be used for the rest of the paper. Most of the results presented here can be found in \cite{Bourbaki}, \cite{Humphreys1} and \cite{Humphreys2}. Let $\mathfrak{g}$ be a semisimple Lie algebra over $\mathbb{C}$  with Cartan subalgebra $\mathfrak{h}.$ Set $l=\rank \mathfrak{g}$ and fix an ordered base $\Pi=\{\alpha_1,\ldots,\alpha_l\}$ of the corresponding root system $\Phi=\Phi^+\sqcup \Phi^-,$ where $\Phi^+$ and $\Phi^-$ denote the sets of positive and negative roots of $\Phi,$ respectively. To each root $\alpha\in \Phi$ corresponds a $1$-dimensional subspace $\mathfrak{g}_{\alpha}$ of $\mathfrak{g}$ (called a \emph{root space}) defined by  
$$
\mathfrak{g}_{\alpha}=\{x\in \mathfrak{g} : [h,x]=\alpha(h)x \mbox{ for all $h\in \mathfrak{h}$}\}.
$$
It is quite common to consider a basis $\mathscr{B}=\{y_{\alpha},h_r,x_{\alpha}: \alpha \in \Phi^+, 1\leqslant r\leqslant l\}
$ for $\mathfrak{g},$ where  $x_{\alpha}\in \mathfrak{g}_{\alpha},$ $y_{\alpha}\in \mathfrak{g}_{-\alpha}$ are root vectors for $\alpha \in \Phi^+$ and $h_r=[x_{\alpha_r},y_{\alpha_r}]$ for $1\leqslant r\leqslant l.$ Such a basis can be chosen in many ways. For example, a standard Chevalley basis (see \cite[Chapter 4]{Carter}) has integral structure constants and hence is easy to work with. (For our purpose though, it is not necessary to make such a refined choice for a basis of $\mathfrak{g}.$) Fixing an ordering $\leqslant$ on $\Phi^+=\{\gamma_1,\ldots,\gamma_m\}$ (with  $\gamma_r=\alpha_r$ for $1\leqslant r\leqslant l$) yields the existence of  an \emph{ordered} basis  
\begin{equation}
\mathscr{B}=\{y_1,\ldots,y_m,h_1,\ldots,h_r,x_1,\ldots,x_m  \}
\label{A_basis_for_g}
\end{equation}
for $\mathfrak{g},$ where  $x_r\in \mathfrak{g}_{\gamma_r},$ $y_{r}\in \mathfrak{g}_{-\gamma_r}$ are root vectors for $1\leqslant r\leqslant m$ and $h_r=[x_{r},y_{r}]$ for $1\leqslant r\leqslant l.$ Throughout this paper, we  fix an ordered basis as in \eqref{A_basis_for_g} for any semisimple Lie algebra $ \mathfrak{g}.$

\subsection{Integral weights}     

The root system $\Phi$ of $\mathfrak{g}$ spans a $\mathbb{Q}$-form $E_0$ of the dual space $\mathfrak{h}^*$ on which the Killing form $(-,-)$ is non-degenerate, providing $E=E_0\otimes_{\mathbb{Q}} \mathbb{R}$ with a natural structure of euclidean space. The $\mathbb{Z}$-span of $\Phi$ in $E$ is called the \emph{root lattice} of $\Phi,$ and the dual  lattice to $\Phi$ in $E,$ defined by 
$$
\Lambda = \{\lambda \in E :\langle \lambda, \alpha \rangle \in \mathbb{Z} \mbox{ for every } \alpha\in \Pi \},
$$ 
is called the \emph{integral weight lattice} associated to $\Phi.$ (Here we adopt the notation $\langle x,y\rangle = 2(x, y) (y,y)^{-1}$ for $x,y \in E$ with $y\neq 0.$) It is a free abelian group of rank $l$ with basis $\{\lambda_1,\ldots,\lambda_l\},$ where $\lambda_1,\ldots,\lambda_l$ denote the \emph{fundamental weights} corresponding to our choice of base $\Pi,$ that is $\langle \lambda_i,\alpha_j\rangle =\delta_{ij}$ for every $1\leqslant i,j\leqslant l.$  In addition, let 
$$
\Lambda^+=\{\lambda\in \Lambda : \langle \lambda, \alpha_r\rangle \geqslant 0\mbox{ for every }1\leqslant r\leqslant l\}
$$
 be the set of \emph{dominant integral weights} and recall the existence of a partial order $\preccurlyeq$ on $\Lambda$, defined by $\mu\preccurlyeq \lambda$ if and only if $\lambda-\mu \in \Gamma,$ where $\Gamma \subset \Lambda$ is the monoid of $\mathbb{Z}_{\geqslant 0}$-linear combinations of simple roots. (We also write $\mu\prec \lambda$ to indicate that $\mu\preccurlyeq \lambda$ and $\mu\neq \lambda.$) Finally, for $\alpha\in \Phi,$ define the reflection $s_{\alpha}:E \to E$ relative to $\alpha$ by 
$$
s_{\alpha}(\lambda)=\lambda-\langle \lambda,\alpha\rangle\alpha,
$$
this for every $\lambda\in \mathfrak{h}^*,$ and denote by $\mathscr{W}$ the finite group $\langle s_{\alpha_r}:1\leqslant r\leqslant l\rangle,$ called the \emph{Weyl group of} $\Phi.$ We say that $\lambda,\mu\in \mathfrak{h}^*$ are \emph{conjugate under the action of $\mathscr{W}$} (or $\mathscr{W}$\emph{-conjugate}) if there exists $w\in \mathscr{W}$ such that $w \lambda=\mu.$ One easily shows that $\mathscr{W}$ stabilizes $\Lambda$ and it is well-known (see \cite[Section 13.2, Lemma A]{Humphreys1}, for example) that each weight in $\Lambda$ is $\mathscr{W}$-conjugate to a unique dominant integral weight. Also  if $\lambda\in \Lambda^+,$ then $w\lambda \preccurlyeq \lambda$ for every  $w\in \mathscr{W}.$

\subsection{Universal enveloping algebras}     

In this section, we recall some elementary facts on universal enveloping algebras of Lie algebras. Most of the results presented here can be found in \cite{Bourbaki}  or \cite{Humphreys1}.  A \emph{universal enveloping algebra} of an arbitrary Lie algebra  $\mathfrak{L}$ over $\mathbb{C}$ is a pair $(\mathfrak{U}(\mathfrak{L}),\iota),$ where $\mathfrak{U}(\mathfrak{L})$ is an associative algebra with $1$ over $\mathbb{C},$ $\iota : \mathfrak{L} \to \mathfrak{U}(\mathfrak{L})$ is a linear map satisfying 
\begin{equation}
\iota([x,y])=\iota(x)\iota(y)-\iota(y)\iota(x),
\label{Universal_enveloping_algebra_:_iota}
\end{equation}
for $x,y\in \mathfrak{L},$ and such that the following universal property holds: for any associative algebra  $\mathfrak{U}$ with $1$ and any linear map $\eta:\mathfrak{L} \to \mathfrak{U}$ satisfying \eqref{Universal_enveloping_algebra_:_iota}, there exists a unique morphism of algebras $\phi : \mathfrak{U}(\mathfrak{L}) \to \mathfrak{U}$ such that $\phi\circ \iota = \eta.$ The existence and uniqueness (up to isomorphism) of such a pair $(\mathfrak{U}(\mathfrak{L}),\iota)$ are not too difficult to establish (see \cite[Section 17.2]{Humphreys1}, for example) and the well-known \emph{Poincar\'e-Birkhoff-Witt Theorem}  (or \emph{PBW-Theorem}) implies that if $\mathfrak{L}$ is a Lie algebra with corresponding universal enveloping algebra $(\mathfrak{U}(\mathfrak{L}),\iota),$ then $\iota$ is injective. Furthermore, if $\mathfrak{L}$ is identified with its image in $\mathfrak{U}(\mathfrak{L})$ and if $(x_1,x_2,\ldots)$ is an ordered basis for $\mathfrak{L},$ then a basis for $\mathfrak{U}(\mathfrak{L})$ is given by 
$$
\left\{x_1^{t_1} \cdots x_k^{t_k} : k\in \mathbb{Z}_{\geqslant 0},~t_1,\ldots,t_k \in \mathbb{Z}_{\geqslant 0}\right\}.
$$

In the case where $\mathfrak{g}$ is a semisimple Lie algebra with ordered basis as in \eqref{A_basis_for_g}, one deduces that a basis for $\mathfrak{U}(\mathfrak{g})$ consists of elements of the form   
\begin{equation}
y_1^{r_1}\cdots y_m^{r_m} h_1^{s_1} \cdots h_l^{s_l}x_1^{t_1}\cdots x_m^{t_m},
\label{Standard_PBW_ordering}
\end{equation}
where $r_i,s_j,t_i\in \mathbb{Z}_{\geqslant 0}$ for every $1\leqslant i\leqslant m$ and every $1\leqslant j\leqslant l.$ Finally, it turns out that $\mathfrak{U}(\mathfrak{g})$ can be decomposed into a direct sum of subspaces of the form $\mathfrak{U}(\mathfrak{g})_{\gamma},$ where $\gamma \in \mathbb{Z}\Phi$ and $\mathfrak{U}(\mathfrak{g})_{\gamma}$ is spanned by those monomials in  \eqref{Standard_PBW_ordering} for which $\gamma =\sum_{i=1}^m{(t_i-r_i)\gamma_i}.$

\subsection{Representations of $\mathfrak{U}(\mathfrak{g})$}                                                              

In this section, we recall some basic properties of $\mathfrak{U}(\mathfrak{g})$-modules (or equivalently, $\mathfrak{g}$-modules). Unless specified otherwise, the results recorded here can be found in \cite[Section 20]{Humphreys1}. Let  $V$ denote an arbitrary $\mathfrak{U}(\mathfrak{g})$-module and for $\mu \in \mathfrak{h}^*,$ set 
$$
V_\mu=\{v\in V: h v=\mu(h) v \mbox{ for all $h\in \mathfrak{h}$}\}.
$$
An element $\mu \in \mathfrak{h}^*$ with $V_\mu\neq 0$ is called a \emph{weight of} $V$ and $V_\mu$ is said to be its corresponding \emph{weight space}. The dimension of $V_\mu$ (possibly infinite) is called the \emph{multiplicity of $\mu$ in} $V$ and is denoted by $\m_V(\mu).$ It behaves well with respect to short exact sequences, in the following sense: if $0\to V_1 \to V_2 \to V_3\to 0$ is a short exact sequence of $\mathfrak{U}(\mathfrak{g})$-modules and $\mu \in \mathfrak{h}^*,$ then 
\begin{equation}
\m_{V_2}(\mu)=\m_{V_1}(\mu)+\m_{V_3}(\mu).
\label{ses}
\end{equation}

Also write $\Lambda(V)$ to denote the set of weights of $V$ and as in the integral case, define a partial order on the latter by saying that $\mu\in \Lambda(V)$ is \emph{under} $\lambda \in \Lambda(V)$ (written $\mu \preccurlyeq \lambda$) if and only if $\lambda-\mu\in \Gamma,$ where $\Gamma$ denotes the monoid of $\mathbb{Z}_{\geqslant 0}$-linear combinations of simple roots. In addition, we write $\mu\prec \lambda$ to indicate that $\mu$ is \emph{strictly}  under $\lambda,$ i.e. $\mu$ is under $\lambda$ and $\mu\neq \lambda.$  
\vspace{5mm}

A $\mathfrak{U}(\mathfrak{g})$-module $V$ is said to be a \emph{weight module} if it is $\mathfrak{h}$-semisimple, that is, if it can be decomposed into a direct sum of its weight spaces. If $\dim V < \infty,$ then $V$ is always a weight module, while if on the other hand $V$ is infinite-dimensional, then the sum of its weight spaces might be a proper submodule. Nevertheless, two weight spaces corresponding to different weights always intersect trivially, from which one easily deduces that if $U,$ $W$ are two submodules of a weight module $V$ and $\mu\in \Lambda(V),$ then 
\begin{equation}
(U+W)_{\mu} = U_{\mu}+W_{\mu}.
\label{(U+W)_mu=?}
\end{equation}

A non-zero vector $v^+\in V$ is called a \emph{maximal vector} of weight $\lambda\in \mathfrak{h}^*$ if $v^+\in V_{\lambda}$ and $x_r v^+=0$ for every $1\leqslant r\leqslant m.$ Also, we say that $V$ is a \emph{highest weight module} of weight $\lambda$ if there exists a maximal vector $v^+\in V_{\lambda}$ such that $\mathfrak{U}(\mathfrak{g}) v^+ =V.$ Write $\mathfrak{n} = \bigoplus_{r=1}^m \langle x_r\rangle_{\C}$ and $\mathfrak{n}^- = \bigoplus_{r=1}^m \langle y_r\rangle_{\C}.$  Since $\mathfrak{U}(\mathfrak{g})=\mathfrak{U}(\mathfrak{n^{-}}) \mathfrak{U}(\mathfrak{h}) \mathfrak{U}(\mathfrak{n}),$  the module $V$ is generated by $v^+$ as a $ \mathfrak{U}(\mathfrak{n^-}) $-module, so that
\begin{equation}
V_{\mu}=\left\langle y_1^{r_1}\cdots y_m^{r_m}v^+ : r_1,\ldots,r_m\in \mathbb{Z}_{\geqslant 0},~\sum_{i=1}^m{r_i \gamma_i}=\lambda-\mu \right\rangle_\mathbb{C}
\label{weight spaces in highest weight modules}
\end{equation}
for any $\mu \in \mathfrak{h}^*.$ Finally, the natural action of the Weyl group $\mathscr{W}$ on $\mathfrak{h}^*$ induces an action on $\Lambda(V).$ As in the integral case, we say that $\lambda,\mu\in \Lambda(V)$ are \emph{conjugate under the action of $\mathscr{W}$} (or $\mathscr{W}$\emph{-conjugate}) if there exists $w\in \mathscr{W}$ such that $w \lambda=\mu.$

\subsection{Verma modules and the BGG category $\mathcal{O}$}     
\label{Subsection_on_Verma_modules}                                 

In the remainder of this paper, we shall be particularly interested in finitely generated, $\mathfrak{h}$-semisimple $\mathfrak{U}(\mathfrak{g})$-modules $V$ such that for every $v\in V,$ the subspace $\mathfrak{U}(\mathfrak{n})v$ of $V$ is finite-dimensional. (The latter condition is called \emph{local $\mathfrak{n}$-finiteness}.) Such modules form the objects of a subcategory $\mathcal{O}$ of the category of (left) $\mathfrak{U}(\mathfrak{g})$-modules,  called the \emph{BGG category}. The latter is closed under submodules, quotients and finite direct sums and it turns out that every irreducible module in $\mathcal{O}$ can be obtained as the quotient of a certain highest weight module, called a \emph{Verma module}. All results presented here can be found in \cite[Chapter 1]{Humphreys2}, for example.

\begin{defin}[Verma module]
Set $\mathfrak{b}=\mathfrak{n} \oplus \mathfrak{h}$ and for $\lambda\in \mathfrak{h}^*,$ let $\mathbb{C}_{\lambda}$ denote the $\mathfrak{U}(\mathfrak{b})$-module defined by $\mathfrak{n}\xi=0$ and $h \xi=\lambda(h) \xi$ for all  $h\in \mathfrak{h}$ and $\xi\in \mathbb{C}.$ The \emph{Verma module  of weight $\lambda$} is the $\mathfrak{U}(\mathfrak{g})$-module $\Delta(\lambda)$ obtained by inducing $\C_\lambda$ from $\mathfrak{b}$ to $\mathfrak{g},$ that is, 
$$
\Delta(\lambda)=\mathfrak{U}(\mathfrak{g})\otimes_{\mathfrak{U}(\mathfrak{b})} \mathbb{C}_{\lambda}.
$$ 
\end{defin}

The module $\Delta(\lambda)$ also admits a description by generators and relations, from which one deduces that $\Delta(\lambda)$ plays the role of  universal highest weight module of weight $\lambda$ in the category $\mathcal{O}$ (see \cite[Section 1.3]{Humphreys2}). Therefore by \eqref{weight spaces in highest weight modules}, one gets that the weight space $\Delta(\lambda)_\mu$ corresponding to $\mu\in \mathfrak{h}^*$ is spanned by the set
\begin{equation}
\mathscr{B}(\lambda)_{\mu}=\left\{y_1^{r_1} \cdots y_m^{r_m}v^{\lambda} : r_1,\ldots,r_m \in \mathbb{Z}_{\geqslant 0},~\sum_{i=1}^{m}{r_i\gamma_i}=\lambda-\mu\right\},
\label{B(lambda)_mu}
\end{equation}
where $v^{\lambda}$ denotes a maximal vector of weight $\lambda$ in $\Delta(\lambda).$ The cardinality of the set \eqref{B(lambda)_mu} equals $P(\lambda-\mu),$ where $P:\mathfrak{h}^* \to \mathbb{Z}_{\geqslant 0}$ corresponds to the \emph{Kostant function}, whose value at $\alpha\in \mathfrak{h}^*$ is defined to be the number of distinct sets of non-negative integers $c_1,\ldots,c_m$ for which $\alpha=\sum_{r=1}^m {c_{r}\gamma_r}.$  The following result consists in a description of a basis for $\Delta(\lambda)_\mu,$ thus leading to the knowledge of the multiplicity of $\mu$ in $\Delta(\lambda).$ Its proof immediately follows from the fact that $\Delta(\lambda) \cong \mathfrak{U}(\mathfrak{n^-})$ as $\mathfrak{U}(\mathfrak{n^-})$-modules (see \cite[Section 20.3]{Humphreys1}, for example).

\begin{lem}\label{A_basis_for_M(lambda)_mu}
Let $\lambda,\mu \in \mathfrak{h}^*$  and consider the Verma module $\Delta(\lambda)$ of weight $\lambda.$ Then the set \eqref{B(lambda)_mu} forms a basis for $\Delta(\lambda)_\mu.$ In particular  $\dim \Delta(\lambda)_{\mu} = P(\lambda-\mu).$
\end{lem}

\begin{remark}\label{Remark_on_weight_spaces_in_Verma_modules}
Let $\lambda, \delta \in \mathfrak{h}^*$ and fix two maximal vectors  $v^{\lambda},v^{\delta}$ in $\Delta(\lambda)_{\lambda},$ $\Delta(\delta)_{\delta},$ respectively. Also let $\gamma \in \Gamma$ and set $\mu=\lambda-\gamma,$ $\nu=\delta-\gamma.$ By Lemma \ref{A_basis_for_M(lambda)_mu}, we have $\m_{\Delta(\lambda)}(\mu)=\m_{\Delta(\delta)}(\nu)$ and the sets $\mathscr{B}(\lambda)_{\mu},$ $\mathscr{B}(\delta)_{\nu}$ as in \eqref{B(lambda)_mu} form ordered bases of $\Delta(\lambda)_{\mu},$ $\Delta(\delta)_{\nu},$ respectively. Furthermore, for any $y\in \mathfrak{U}(\mathfrak{n}^-),$ the coefficients of $yv^{\lambda}$ with respect to $\mathscr{B}(\lambda )_{\mu}$  and the coefficients of $yv^{\delta}$ with respect to $\mathscr{B}(\delta )_{\nu}$   are identical, since obtained by successively applying standard commutation formulas in $\mathfrak{U}(\mathfrak{g}).$
\end{remark}

It turns out that $\Delta(\lambda)$ contains a unique maximal submodule $\rad(\lambda)$ and throughout this paper, we write $L(\lambda)=\Delta(\lambda)/\rad(\lambda)$ for the corresponding irreducible quotient. Unfortunately, there is no analogue of Lemma \ref{A_basis_for_M(lambda)_mu} for weight spaces in $L(\lambda)$ for an arbitrarily given $\lambda\in \mathfrak{h}^*.$ Nevertheless, applying \eqref{ses} to the short exact sequence $$0 \to \rad(\lambda)\to \Delta(\lambda)\to L(\lambda)\to 0$$ and using Lemma \ref{A_basis_for_M(lambda)_mu} , one easily sees that knowing the multiplicity of $\mu\in \mathfrak{h}^*$ in $\rad(\lambda)$ leads to the knowledge of $\m_{L(\lambda)}(\mu)$ as well.  Now for $\lambda\in \mathfrak{h}^*$ arbitrary, no simple description of $\rad(\lambda)$ is known, while in the case where $\lambda$ is dominant integral, then the following result gives a better understanding of the structure of $\rad(\lambda).$

\begin{thm}\label{A_description_of_N(lambda)}
Let $\lambda =\sum_{r=1}^l{a_r\lambda_r}\in \Lambda^+$ be a dominant integral weight and fix a maximal vector $v^{\lambda}$ of weight $\lambda$ in $\Delta(\lambda).$  Then the following assertions hold.
\begin{enumerate}
\item \label{description_of_N(lambda):part1}For every $1\leqslant r\leqslant l,$ the element $y_{r}^{a_r+1}v^\lambda$ is a maximal vector of weight $\lambda-(a_r+1)\alpha_r$ in $\Delta(\lambda).$
\item \label{description_of_N(lambda):part2}For every $1\leqslant r\leqslant l,$ the $\mathfrak{U}(\mathfrak{g})$-module generated by $y_r^{a_r+1}v^\lambda$ is isomorphic to $\Delta(\lambda-(a_r+1)\alpha_r).$
\item \label{description_of_N(lambda):part3}The unique maximal submodule $\rad(\lambda)$ of $\Delta(\lambda)$ is given by $$
\rad(\lambda)=\sum_{r=1}^{l}{\mathfrak{U}(\mathfrak{g})y_r^{a_r+1}v^\lambda}.
$$
\end{enumerate}
\end{thm}

\begin{proof}
The proof of Parts \ref{description_of_N(lambda):part1} and \ref{description_of_N(lambda):part2} essentially depends on some standard commutation formulas in $\mathfrak{U}(\mathfrak{g})$ (see \cite[Proposition 1.4]{Humphreys2}, for example). Also, we refer the reader to \cite[Theorem 2.6]{Humphreys2} for a proof of Part \ref{description_of_N(lambda):part3}.
\end{proof}

\section{Proof of Proposition \ref{Main_Result_1}}     
\label{Proof_of_the_first_result}                      

For a non-empty subset $J$  of $\{1,\ldots,l\},$ set $\mathfrak{h}_{J}= \langle h_{ j} : j \in J\rangle$ as well as  $\mathfrak{g}_{J}=\langle \mathfrak{g}_{\pm \alpha_j} : j \in J\rangle.$ Clearly the Levi subalgebra $\mathfrak{g}_{J}$ of $\mathfrak{g}$ is a semisimple Lie algebra over $\mathbb{C},$ having Cartan subalgebra $\mathfrak{h}_{J}$ and root system $\Phi_{J}=\Phi \cap \mathbb{Z}\{\alpha_j\}_{j\in {J}}.$ We start by stating the following result, whose proof can be found in \cite[Lemma 2.2.8]{BGT}, for example.

\begin{lem}\label{Restriction_to_suitable_Levi_subalgebra}
Let $\lambda\in \Lambda^+$ be a dominant integral weight and let $\mu\in \Lambda$ be such that $\mu\prec \lambda.$ Also let $J\subseteq \{1,\ldots,l\}$ be such that $\mu=\lambda-\sum_{j\in J}{c_j \alpha_j} $ for some subset $\{c_j\}_{j\in J}$ of $\Z_{\geqslant 0}.$ Then 
$$
\m_{L(\lambda)}(\mu)=\m_{L\left(\lambda|_{\mathfrak{h}_{J}}\right)}\left(\mu|_{\mathfrak{h}_{J}}\right).
$$ 
\end{lem}

\begin{remark}\label{A_first_simplification_of_J_in_Prop_1}
Let $\lambda=\sum_{r=1}^l{a_r \lambda_r}\in \Lambda^+$ and let $\mu\in \Lambda$ be such that  $ \mu = \lambda - \sum_{r=1}^l{ c_r \alpha_r}\in \Lambda$ for some $c_1,\ldots,c_l\in \mathbb{Z}_{\geqslant 0}.$ Also let $J\subseteq  \{1,\ldots,l\}$ be minimal such that $\mu=\lambda-\sum_{j\in J}{c_j\alpha_j}.$ An application of Lemma \ref{Restriction_to_suitable_Levi_subalgebra} then shows that $\m_{L(\lambda)}(\mu)$ is independent of the value of each $a_k$ such that $k\notin J.$ In particular  
$
\m_{L(\lambda)}(\mu)=\m_{L (\lambda')}(\mu'),
$ 
where $\lambda'=\lambda-\sum_{k\notin J}{a_k\lambda_k}$ and $\mu'=\lambda'-(\lambda-\mu).$ Consequently one can assume $c_j\neq 0$ for every $j\in J$ in Proposition \ref{Main_Result_1} and focus on the situation where $\mathfrak{g}$ is simple.
\end{remark}

In order to give a proof of Proposition \ref{Main_Result_1} in the general case, we proceed by induction on the cardinality of the subset $J$ introduced  in the statement of the latter. The difficulty mainly resides in dealing with the base case of the induction, that is, with the situation where  $J $ is a singleton. For $\lambda =\sum_{r=1}^l{a_r\lambda_r}\in \Lambda^+$ a dominant integral weight, $1\leqslant j\leqslant l$ and $x\in \mathbb{Z},$ define 
$$
\lambda_{j,x}=\lambda+(x-a_j)\lambda_j.$$
Clearly $\lambda_{j,x}$ is simply obtained from $\lambda$ by replacing $a_j$ with $x$ and hence remains dominant integral if and only if $x\in \mathbb{Z}_{\geqslant 0}.$ In addition, for $\mu\in \Lambda$ such that $\mu\preccurlyeq \lambda,$ also  define
$$
\mu^{\lambda_{j,x}}=\lambda_{j,x} - (\lambda-\mu).
$$
Obviously $\lambda^{\lambda_{j,x}}=\lambda_{j,x}$ and if $\mu=\lambda-\sum_{r=1}^l{c_r\alpha_r},$ then $\mu^{\lambda_{j,x}} = \lambda_{j,x}-\sum_{r=1}^l{c_r\alpha_r}.$ In other words, $\mu^{\lambda_{j,x}}$ is the unique integral weight for which $\lambda-\mu =\lambda_{j,x}-\mu^{\lambda_{j,x}}.$  Finally,  define the \emph{$j$-level} of $\alpha=\sum_{r=1}^l{d_r\alpha_r}\in \Gamma$ to be  $\level_j(\alpha)=d_j.$ The following elementary result provides the reader with a way of becoming familiar with the recently introduced notation. It shall be used in Section \ref{Section:New_Freudenthal's_formula}.

\begin{lem}\label{Another_expression_for_(lambda_j,x,alpha)_for_every_alpha_in_Gamma}
Let $\lambda\in \Lambda^+$ be a dominant integral weight and let $\alpha\in \Gamma.$ Then for every $1\leqslant j\leqslant l$ and every $x\in \mathbb{Z},$ we have 
$$
(\lambda_{j,x},\alpha)=(\lambda_{j,0},\alpha)+x\level_j(\alpha)(\lambda_j,\alpha_j).
$$
\end{lem}

\begin{proof}
First notice that $(\lambda_{j,x},\alpha)=(\lambda+(x-a_j)\lambda_j,\alpha) =(\lambda-a_j\lambda_j,\alpha) + x(\lambda_j,\alpha)$ by definition of $\lambda_{j,x}$ and by bilinearity of $(-,-).$ Also $\lambda-a_j\lambda_j$ is obviously obtained from $\lambda$ by replacing $a_j$ by $0,$ i.e. $\lambda-a_j\lambda_j=\lambda_{j,0}.$ Finally, observe that $(\lambda_j,\alpha_r)=\delta_{rj}(\lambda_j,\alpha_j)$ for every $1\leqslant r\leqslant l$ and hence $(\lambda_j,\alpha)=\level_j(\alpha)(\lambda_j,\alpha_j),$   completing the proof.
\end{proof}

The following result can sometimes provide one with a way of  comparing certain weight spaces in different Verma modules. Its proof essentially relying on Theorem \ref{A_description_of_N(lambda)}, we shall adopt the notation introduced in the latter.

\begin{lem}\label{An_isomorphism_from_M(lambda)mu_to_M(lambda,j,x)mu,j,x}
Let $\lambda=\sum_{r=1}^l{a_r\lambda_r}\in \Lambda^+$ be a dominant integral weight and let $\mu\in \Lambda$ be such that $\mu=\lambda-\sum_{r=1}^l{c_r \alpha_r}$ for some $c_1,\ldots,c_l\in \mathbb{Z}_{\geqslant 0},$ so that $\mu\preccurlyeq \lambda.$ Also assume the existence of $1\leqslant j\leqslant l$ satisfying $0<c_j\leqslant a_j$ and let $x\in \mathbb{Z}_{\geqslant c_j}.$  Finally, fix two maximal vectors $v^{\lambda},v^{\lambda_{j,x}}$ in $\Delta(\lambda)_\lambda,$ $\Delta(\lambda_{j,x})_{\lambda_{j,x}},$ respectively. Then there exists an isomorphism of vector spaces $\phi:\Delta(\lambda)_{\mu}\longrightarrow \Delta(\lambda_{j,x})_{\mu^{\lambda_{j,x}}}$ such that 
\begin{equation}
 \phi \left( (\mathfrak{U}(\mathfrak{g})y_r^{a_r+1}v^{\lambda})_{\mu} \right)= (\mathfrak{U}(\mathfrak{g})y_r^{a_r+1}v^{\lambda_{j,x}} )_{\mu^{\lambda_{j,x}}} 
\label{Main_equation_from_M(lambda)mu_to_M(lambda,j,x)mu,j,x}
\end{equation}
for every  $1\leqslant r\leqslant l.$ 
\end{lem}

\begin{proof}
Let $1\leqslant j\leqslant l$ and $x\in \mathbb{Z}_{\geqslant c_j}$ be as in the statement of the Lemma.  The set $\mathscr{B}(\lambda)_{\mu}$ as in  \eqref{B(lambda)_mu} forms a basis for $\Delta(\lambda)_{\mu}$ by Lemma \ref{A_basis_for_M(lambda)_mu}, showing the existence of a unique linear map $\phi:\Delta(\lambda)_{\mu}\longrightarrow \Delta(\lambda_{j,x})_{\mu^{\lambda_{j,x}}}$ such that 
\begin{equation}
\phi(y_1^{r_1}\cdots y_m^{r_m} v^{\lambda}) = y_1^{r_1}\cdots y_m^{r_m} v^{\lambda_{j,x}}
\label{Proof_of_iso_between_M(l)_mu_and_M(l,j,x)_mu}
\end{equation}
for every $r_1,\ldots,r_m\in \mathbb{Z}_{\geqslant 0}$ satisfying $\sum_{i=1}^m{r_i\gamma_i} =\lambda-\mu.$ Since the elements on the right-hand  side of \eqref{Proof_of_iso_between_M(l)_mu_and_M(l,j,x)_mu} form a basis for  $\Delta(\lambda_{j,x})_{\mu^{\lambda_{j,x}}}$ by Lemma \ref{A_basis_for_M(lambda)_mu} again, the linear map  $\phi$ is an isomorphism of vector spaces. Now as we assumed $ c_j\leqslant a_j,$ we immediately get that $\mu$ is not under $\lambda-(a_j+1)\alpha_j$ and hence we have
$$
(\mathfrak{U}(\mathfrak{g})y_j^{a_j+1}v^\lambda)_{\mu} =0
$$ 
by Theorem \ref{A_description_of_N(lambda)} (Part \ref{description_of_N(lambda):part2}). Similarly $(\mathfrak{U}(\mathfrak{g})y_j^{a_j+1}v^{\lambda_{j,x}})_{\mu^{\lambda_{j,x}}} =0$ and hence the equality \eqref{Main_equation_from_M(lambda)mu_to_M(lambda,j,x)mu,j,x} holds for $r=j.$ In the remainder of the proof, assume $1\leqslant r\leqslant l$ different from $j$ and set $\delta_r=\lambda-(a_r+1)\alpha_r.$ By Lemma \ref{A_basis_for_M(lambda)_mu} and Theorem \ref{A_description_of_N(lambda)} (Part \ref{description_of_N(lambda):part2}), one successively gets
\begin{align*}
\dim (\mathfrak{U}(\mathfrak{g})y_r^{a_r+1}v^\lambda)_{\mu}
&=	\dim \Delta(\delta_r)_\mu					\cr
	&=	P(\delta_r-\mu)							\cr
						&= 	P\left( \lambda_{j,x}+(\delta_r-\lambda) -(\lambda_{j,x}+\mu-\lambda)\right)									\cr
						&=	P( \delta_r^{\lambda_{j,x}} -\mu^{\lambda_{j,x}})											\cr
						&=	\dim \Delta(\delta_r^{\lambda_{j,x}})_{\mu^{\lambda_{j,x}}}					\cr
&=	\dim (\mathfrak{U}(\mathfrak{g})y_r^{a_r+1}v^{\lambda_{j,x}})_{\mu^{\lambda_{j,x}}},
\end{align*}
and hence it suffices to show that $\phi \left((\mathfrak{U}(\mathfrak{g})y_r^{a_r+1}v^\lambda)_{\mu}\right) \subseteq (\mathfrak{U}(\mathfrak{g})y_r^{a_r+1}v^{\lambda_{j,x}})_{\mu^{\lambda_{j,x}}}.$ Let   $r_1,\ldots,r_m\in \mathbb{Z}_{\geqslant 0}$ be such that $r_1\gamma_1+\cdots + r_m\gamma_m =\mu -(a_r+1)\gamma_r$ and consider $v=y_1^{r_1}\cdots y_m^{r_m}y_r^{a_r+1}v^{\lambda}\in \Delta(\lambda)_{\mu}.$ In addition, write   $\mathscr{B}(\lambda)_{\mu}=\{v_1,\ldots,v_k\},$ so that $\mathscr{B}(\lambda_{j,x})_{\mu^{\lambda_{j,x}}}=\{\phi(v_1),\ldots,\phi(v_k)\}$ and let $\xi_1,\ldots,\xi_m\in \mathbb{C}$ denote the unique complex coefficients satisfying  $v= \xi_1 v_1 + \cdots + \xi_k v_k .$ An application of  Remark \ref{Remark_on_weight_spaces_in_Verma_modules} then yields 
$$
\phi(v)=\phi\left(\sum_{i=1}^k{\xi_iv_i}\right)_{~}= \sum_{i=1}^k{\xi_i\phi(v_i)} = y_1^{r_1}\cdots y_m^{r_m}y_r^{a_r+1}v^{\lambda_{j,x}} \in (\mathfrak{U}(\mathfrak{g})y_r^{a_r+1}v^{\lambda_{j,x}})_{\mu^{\lambda_{j,x}}}.
$$
By Lemma \ref{A_basis_for_M(lambda)_mu} and Theorem \ref{A_description_of_N(lambda)} (Parts \ref{description_of_N(lambda):part1} and \ref{description_of_N(lambda):part2}), any element of  $(\mathfrak{U}(\mathfrak{g})y_r^{a_r+1}v^\lambda)_{\mu}$ can be expressed as a linear combination of vectors such as $v,$ from which the result follows.
\end{proof}

Let $\lambda \in \Lambda^+$ be a dominant integral weight and let $\mu \in \Lambda$ be such that $\mu\prec \lambda.$ Also fix $1\leqslant j\leqslant l$ and  $x\in \mathbb{Z}_{\geqslant 0}.$ Then $\m_{\Delta(\lambda)}(\mu)=P(\lambda-\mu)=P(\lambda_{j,x}-\mu^{\lambda_{j,x}})=\m_{\Delta(\lambda_{j,x})}(\mu^{\lambda_{j,x}})$ by Lemma \ref{A_basis_for_M(lambda)_mu} and applying \eqref{ses} to the exact sequences $0\to \rad(\lambda)\to \Delta(\lambda)\to L(\lambda)\to 0$ and $0\to \rad(\lambda_{j,x})\to \Delta(\lambda_{j,x})\to L(\lambda_{j,x})\to 0$ shows that $\m_{L(\lambda)}(\mu)=\m_{L(\lambda_{j,x})}(\mu^{\lambda_{j,x}})$ if and only if   
\begin{equation}
\m_{\rad(\lambda)}(\mu)=\m_{\rad(\lambda_{j,x})}(\mu^{\lambda_{j,x}}).
\label{sufficient_condition}
\end{equation}
Now since both $\lambda$ and $\lambda_{j,x}$ are dominant integral (as $x\in \mathbb{Z}_{\geqslant 0}$), applying Theorem \ref{A_description_of_N(lambda)} (Part \ref{description_of_N(lambda):part3}) shows that the modules $\rad(\lambda)$ and $\rad(\lambda_{j,x})$ are given by
$$
\rad(\lambda) = \sum_{r=1}^l{\mathfrak{U}(\mathfrak{g})y_r^{a_r+1}v^{\lambda}},\mbox{ }  \rad(\lambda_{j,x}) =  \sum_{r=1}^l{\mathfrak{U}(\mathfrak{g})y_r^{a_r+1}v^{\lambda_{j,x}}}.
$$

Finally, in Lemma  \ref{An_isomorphism_from_M(lambda)mu_to_M(lambda,j,x)mu,j,x}, it was shown that in the case where $1\leqslant j\leqslant l$ satisfies $0<c_j\leqslant a_j$ and $x\in \mathbb{Z}_{\geqslant c_j},$ then there exists an isomorphism of vector spaces $\phi : \Delta(\lambda)_{\mu} \longrightarrow \Delta(\lambda_{j,x})_{\mu^{\lambda_{j,x}}}$ such that the diagram  
\begin{equation*}
\begin{xy}
(0,20)*+{\Delta(\lambda)_{\mu}}="a"; (65,20)*+{\Delta(\lambda_{j,x})_{\mu^{\lambda_{j,x}}}}="b";%
(0,0)*+{\left(\mathfrak{U}(\mathfrak{g})y_r^{a_r+1}v^{\lambda}\right)_\mu}="c"; (65,0)*+{\left(\mathfrak{U}(\mathfrak{g})y_r^{a_r+1}v^{\lambda_{j,x}}\right)_{\mu^{\lambda_{j,x}}}}="d";%
{\ar_{\cong} "a";"b"}?*!/_3mm/{\phi};
{\ar@{^{(}->} "c";"a"};{\ar@{^{(}->} "d";"b"};%
{\ar_{\cong} "c";"d"}?*!/_4mm/{\phi|_{ \left(\mathfrak{U}(\mathfrak{g})y_r^{a_r+1}v^{\lambda}\right)_{\mu}}};%
\end{xy}
\end{equation*}
commutes for every $1\leqslant r\leqslant l$. Therefore \eqref{sufficient_condition} is satisfied thanks to \eqref{(U+W)_mu=?} and hence the following result holds. (Observe that in the case where $c_j=0,$ then the assertion immediately follows from Remark \ref{A_first_simplification_of_J_in_Prop_1}.)

\begin{corol}\label{A_first_simplification} 
Let $\lambda =\sum_{r=1}^l{a_r\lambda_r}\in \Lambda^+$ be a dominant integral weight and let $\mu\in \Lambda$ be such that $\mu=\lambda-\sum_{r=1}^l{c_r \alpha_r}$ for some $c_1,\ldots,c_l\in \mathbb{Z}_{\geqslant 0},$ so that $\mu\preccurlyeq \lambda.$ Also assume the existence of $1\leqslant j\leqslant l$ such that $0 \leqslant c_j\leqslant a_j.$ Then for every  integer $x\in \mathbb{Z}_{\geqslant c_j},$ we have
$$
\m_{L(\lambda)}(\mu)=\m_{L(\lambda_{j,x})}(\mu^{\lambda_{j,x}}).
$$
\end{corol}

We are finally ready to give a proof of Proposition \ref{Main_Result_1}, arguing by induction on the cardinality of the subset $J\subseteq \{1,\ldots,l\}$ defined in the statement of the latter. First observe that in the case where $J$ is a singleton, then an application of Corollary \ref{A_first_simplification} immediately yields the desired result. If on the other hand $|J|>1,$ then fix $j\in J.$ By Corollary \ref{A_first_simplification} again, we get that
\begin{equation}
\m_{L(\lambda)}(\mu)=\m_{L(\lambda_{j,c_j})}(\mu^{\lambda_{j,c_j}}).
\label{Proof_of_Prop_1_ind_step}
\end{equation}
Clearly $\lambda_{j,c_j}\in \Lambda^+,$ $\mu^{\lambda_{j,c_j}} =\lambda_{j,c_j} -\sum_{r=1}^l{c_r\alpha_r}$ and $0\leqslant c_k\leqslant a_k$ for every $k\in K=J-\{j\}.$ In addition, adopting the notation of Theorem \ref{New_Freudenthal's_formula}, one sees that $\lambda_{j,c_j}'=\lambda'$ and $(\mu^{\lambda_{j,c_j}})'=\mu'.$ Therefore 
$$
\m_{L(\lambda_{j,c_j})}(\mu^{\lambda_{j,c_j}})=\m_{L(\lambda')}(\mu')
$$
thanks to our induction assumption, which together with \eqref{Proof_of_Prop_1_ind_step}, completes the proof.

\section{Proof of Theorem \ref{New_Freudenthal's_formula}}     
\label{Section:New_Freudenthal's_formula}					  

In this section, we give a proof of Theorem \ref{New_Freudenthal's_formula}, which basically consists in a modified version of the usual formula of Freudenthal (see Theorem \ref{Freudenthal's formula} below), applicable in certain situations. For $\lambda\in \Lambda^+$ a dominant integral weight and $\mu\prec \lambda,$ set 
$$
\dlm(\lambda,\mu)=2(\lambda+\rho,\lambda-\mu)-||\lambda-\mu||^2,
$$
where $\rho$ denotes the half-sum of all positive roots in $\Phi,$ or equivalently, the sum of all fundamental weights. The following formula, due to Freudenthal, gives a recursive way to compute the multiplicity of $\mu$ in $L(\lambda).$ We refer the reader to \cite[Theorem 22.3]{Humphreys1} for more details.

\begin{thm}[Freudenthal's Formula]\label{Freudenthal's formula}
Let $\lambda\in \Lambda^+$ be a dominant integral weight. Then the multiplicity in $L(\lambda)$ of any weight $\mu\in \Lambda$  is given recursively by  
$$
\dlm(\lambda,\mu) \m_{L(\lambda)}(\mu)=2 \sum_{r=1}^{\infty}{\sum_{\alpha \in \Phi^+}{\m_{L(\lambda)}(\mu+r\alpha) (\mu+r\alpha,\alpha)}}.
$$
\end{thm}

\begin{remark}
Thanks to \cite[Lemma 13.4 (C) and Proposition 21.3]{Humphreys1}, one gets that $\m_{L(\lambda)}(\mu)=0$ if $\dlm(\lambda,\mu)= 0.$ (In particular, this implies that $\dlm(\lambda,\mu)\neq 0$ if $\mu\preccurlyeq \lambda$ is dominant integral.) Therefore Theorem \ref{Freudenthal's formula} provides an effective method for computing weight multiplicities inside a given irreducible finite-dimensional module.
\end{remark}

Write $\mathbb{Q}_{\leqslant 1}[X]$ for the set of all linear polynomials in the indeterminate $X$ with coefficients in $\mathbb{Q}.$
In a similar fashion, let $\mathbb{Q}_1[X]$ denote the subset of all \emph{non-constant} linear polynomials in $\mathbb{Q}_{\leqslant 1}[X].$ The following preliminary result is elementary and could easily be stated in a much more general setting. Nevertheless, it is included in its most simple form for clarity.

\begin{lem}\label{Elementary_result_on_divisibility_between_polynomials}
Let $f\in \mathbb{Q}_{\leqslant 1}[X],$ $g\in \mathbb{Q}_1[X]$ be two linear polynomials (with $g$ non-constant) and assume the existence of a non-zero integer $x \in \mathbb{Z}$ such that $g(x)g(x+1)\neq 0$ and $f(x)g(x)^{-1}=f(x+1)g(x+1)^{-1}\in \mathbb{Z}.$ Then $f$ is an integral multiple of $g.$ In particular, either $f=0$ or $f\in \mathbb{Q}_1[X]$ as well. 
\end{lem}

\begin{proof}
Let $f,g,x$ be as in the statement of the Lemma and write $f=aX+b,$ $g=cX+q,$ where $c\neq 0.$ By assumption, $(ax+b)(c(x+1)+q)=(cx+q)(a(x+1)+b),$ which  one easily sees, translates to $aq-bc=0.$ If $a=0,$ then one immediately gets $bc=0,$ which implies $f=0$ since $g$ is non-constant (in which case the assertion trivially holds). Therefore assume $a\neq 0$ in the remainder of the proof, so that replacing $q$ by $a^{-1}bc$ yields
$$
\frac{f(x)}{g(x)}=\frac{a}{c} \in \mathbb{Z},
$$
that is, $a$ is an integral multiple of $c.$ Substituting $a$ by $rc$ for some non-zero integer $r\in \mathbb{Z},$ one gets that $aq-bc=c(rq-b)=0,$ so that $b=rq$ (as $c\neq 0).$ Therefore $f=rg$ as desired, thus completing the proof.
\end{proof}

The overall structure of the proof of Theorem \ref{New_Freudenthal's_formula} is as follows. First we show the existence of $N\in \mathbb{Z}$ and two non-zero polynomials $f=aX+b\in \mathbb{Q}_{\leqslant 1}[X],$ $g=cX+q\in \mathbb{Q}_1[X]$ such that $\m_{L(\lambda)}(\mu)= f(x)  g(x)^{-1} $ for every $x\in \mathbb{Z}_{\geqslant N},$ where we explicitly determine $a,c\in \Q$ (see Lemmas \ref{The_polynomial_g} and \ref{The_polynomial_f}). By Lemma  \ref{Elementary_result_on_divisibility_between_polynomials}, we then get that $f$ is an integral multiple of $g,$ i.e. $f=\m_{L(\lambda)}(\mu) g,$ and thus the multiplicity of $\mu$ in $L(\lambda)$ equals $ac^{-1},$ from which we easily conclude.

\begin{lem}\label{The_polynomial_g}
Let $\lambda=\sum_{r=1}^l{a_r\lambda_r}\in \Lambda^+$ be a dominant integral weight and let $\mu\in \Lambda$ be such that $\mu = \lambda- \sum_{r=1}^l{c_r\alpha_r} $ for some $c_1,\ldots,c_l\in \mathbb{Z}_{\geqslant 0}.$ Also fix $1\leqslant j\leqslant l$ and set 
\begin{equation}
c=2c_j(\lambda_j,\alpha_j) \in \Q.
\label{explicit_value_of_c}
\end{equation}  
Then adopting the notation introduced in Section  \textnormal{\ref{Proof_of_the_first_result}}, there exists $q\in \mathbb{Q}$ such that for every $x \in \mathbb{Z}_{\geqslant 0},$ the linear polynomial $g=cX + q \in \mathbb{Q}_{\leqslant 1}[X]$ satisfies the equality 
$$
\dlm(\lambda_{j,x},\mu^{\lambda_{j,x}})=g(x).
$$ 
Moreover,  $g\in \mathbb{Q}_1[X]$ if and only if $\level_j(\lambda-\mu)\neq 0,$ or equivalently, if and only if $c_j\neq 0.$
\end{lem}

\begin{proof}
Fix $x\in \mathbb{Z}_{\geqslant 0}.$ By definition, we have  $\lambda_{j,x}-\mu^{\lambda_{j,x}}=\lambda-\mu,$ from which one easily deduces that
$$
\dlm  (\lambda_{j,x},\mu^{\lambda_{j,x}}) = 2\left(\lambda_{j,x},\lambda -\mu \right)+(2\rho-(\lambda-\mu),\lambda-\mu).
$$
Now $(\lambda_{j,x},\lambda-\mu)= (\lambda_{j,0},\lambda-\mu) + x \level_j(\lambda-\mu)(\lambda_j,\alpha_j)=(\lambda_{j,0},\lambda-\mu)+c_j(\lambda_j,\alpha_j)x$ by Lemma \ref{Another_expression_for_(lambda_j,x,alpha)_for_every_alpha_in_Gamma} and hence setting $q= \left(2\rho + 2 \lambda_{j,0} -(\lambda-\mu) ,\lambda -\mu \right)$ yields
$$
\dlm (\lambda_{j,x},\mu^{\lambda_{j,x}} ) = cx + q,
$$  
with $c\in \Z$ as in \eqref{explicit_value_of_c}. Since $q\in \mathbb{Q}$ is clearly independent of $x$ and since the latter was arbitrarily chosen, we get that the linear polynomial $g =cX + q\in \mathbb{Q}_{\leqslant 1}[X]$ satisfies   $\dlm (\lambda_{j,x},\mu^{\lambda_{j,x}})=g(x)$ for every $x \in \mathbb{Z}_{\geqslant 0}$ as desired. Finally, notice that $g\in \mathbb{Q}_{\leqslant 1}[X]$  is non-constant if and only if  $c=2c_j(\lambda_j,\alpha_j)\neq 0.$ Since $(\lambda_j,\alpha_j)\neq 0$, the second assertion holds as well. 
\end{proof}

\begin{remark}\label{Remark_on_g}
Observe that in the case where $\level_j(\lambda-\mu) \neq 0$ in Lemma \ref{The_polynomial_g}, then the linear polynomial $g\in \mathbb{Q}_{1}[X]$ is strictly increasing in $X.$
\end{remark}

For $1\leqslant j\leqslant l,$  we define a subset $\Phi^+_j$ of $\Phi^+$ by $\Phi^+_j=  \{\alpha \in \Phi^+ : \level_j(\alpha)>0\}.$ Clearly a positive root $\alpha\in \Phi^+$ belongs to  $\Phi_j^+$ if and only if $\alpha_j$ appears in the decomposition of $\alpha$ as a sum of simple roots, or equivalently, if and only if $(\lambda_j,\alpha)\neq 0.$

\begin{lem}\label{The_polynomial_f}
Let $\lambda =\sum_{r=1}^l{a_r\lambda_r}\in \Lambda^+$ be a dominant integral weight and let $\mu\in \Lambda$ be such that $\mu = \lambda- \sum_{r=1}^l{c_r\alpha_r} $ for some $c_1,\ldots,c_l\in \mathbb{Z}_{\geqslant 0}.$ Also assume the existence of $1\leqslant j\leqslant l$ such that $0<c_j \leqslant a_j$ and set 
\begin{equation}
a=2\sum_{r=1}^{c_j}{\sum_{\alpha\in \Phi_j^+}{\level_j(\alpha)\m_{L(\lambda)}(\mu+r\alpha)}(\lambda_j,\alpha_j)} \in \mathbb{Q}.
\label{explicit_value_of_a}
\end{equation}
Then adopting the notation introduced in Section  \textnormal{\ref{Proof_of_the_first_result}}, there exists $b\in \mathbb{Q}$ such that for every $x \in \mathbb{Z}_{\geqslant c_j},$ the linear polynomial $f=a X + b \in \mathbb{Q}_{\leqslant 1}[X]$ satisfies the equality
$$
2\sum_{r=1}^\infty{\sum_{\alpha\in \Phi^+}{\m_{L(\lambda_{j,x})}(\mu^{\lambda_{j,x}}+r\alpha)(\mu^{\lambda_{j,x}}+r\alpha,\alpha )}}=f(x).
$$
\end{lem}

\begin{proof}
Let $1\leqslant j\leqslant l$ be as in the statement of the lemma and fix $x\in \mathbb{Z}_{\geqslant c_j}.$ By  bilinearity of $(-,-),$ we have   $(\mu^{\lambda_{j,x}}+r\alpha,\alpha)  = (\lambda_{j,x},\alpha) + (\mu-\lambda+r\alpha,\alpha)$ and hence Lemma \ref{Another_expression_for_(lambda_j,x,alpha)_for_every_alpha_in_Gamma} yields
$$
(\mu^{\lambda_{j,x}}+r\alpha,\alpha)= \level_j(\alpha) (\lambda_j,\alpha_j)x  + (\lambda_{j,0} + \mu-\lambda+r\alpha,\alpha).
$$ 
On the other hand, an application of Proposition \ref{Main_Result_1} shows that $\m_{L(\lambda_{j,x})}(\mu^{\lambda_{j,x}}+r\alpha)=\m_{L(\lambda)}(\mu+ r\alpha)$ for every $r\in \mathbb{Z}_{\geqslant 0}$ and every $\alpha \in \Phi^+.$ Setting $b= 2\sum_{r=1}^{\infty}{\sum_{\alpha \in \Phi^+}{\m_{L(\lambda)}(\mu+r\alpha)(\lambda_{j,0} + \mu-\lambda+r\alpha,\alpha)}}$ then yields
$$
2\sum_{r=1}^{\infty}{\sum_{\alpha \in \Phi^+ }{\m_{L(\lambda_{j,x})}(\mu^{\lambda_{j,x}}+r\alpha)(\mu^{\lambda_{j,x}}+r\alpha,\alpha)}} = 2\left(\sum_{r=1}^{\infty}{\sum_{\alpha \in \Phi^+ }{\m_{L(\lambda)}(\mu + r\alpha )\level_j(\alpha)(\lambda_j,\alpha_j)}}\right)x +b.
$$ 
Now $\level_j(\alpha)=0$ if $\alpha \notin \Phi^+_j,$ and if $r>c_j$ and $\alpha\in \Phi^+_j,$ then $\mu+r\alpha$ is not under $\lambda.$  Therefore we have
$$
2\sum_{r=1}^{\infty}{\sum_{\alpha \in \Phi^+ }{\m_{L(\lambda_{j,x})}(\mu^{\lambda_{j,x}}+r\alpha)(\mu^{\lambda_{j,x}}+r\alpha,\alpha)}} = ax +b,
$$ 
with $a\in \mathbb{Q}$ as in \eqref{explicit_value_of_a}. Finally, since $b\in \mathbb{Q}$ is independent of $x$ and since the latter was arbitrarily chosen, we get that the linear polynomial $f =a X + b \in \mathbb{Q}_1[X]$ satisfies the desired condition.
\end{proof}

We are now able to give a proof of Theorem \ref{New_Freudenthal's_formula}. Let $\lambda,\mu,j$ be as in the statement of the latter and first observe that $\m_{L(\lambda)}(\mu)=\m_{L(\lambda_{j,x})}(\mu^{\lambda_{j,x}})$ for every $x\in \mathbb{Z}_{\geqslant c_j}$ by Proposition \ref{Main_Result_1}. Now thanks to Lemma \ref{The_polynomial_g} and Remark \ref{Remark_on_g}, we get the existence of a linear polynomial $g=cX+q\in \mathbb{Q}_{1}[X]$ as well as $N\in \mathbb{Z}_{\geqslant 0}$ such that $g(x)=\dlm(\lambda_{j,x},\mu^{\lambda_{j,x}})\neq 0$ for   $x\geqslant \max \{c_j,N\}.$ An application of Theorem \ref{Freudenthal's formula} then yields 
$$
\m_{L(\lambda)}(\mu) = \frac{2}{g(x)}\sum_{r=1}^\infty{\sum_{\alpha\in \Phi^+}{\m_{L(\lambda_{j,x})}(\mu^{\lambda_{j,x}}+r\alpha)(\mu^{\lambda_{j,x}}+r\alpha,\alpha )}}
$$
for every $x\geqslant \max \{c_j,N\},$ which by Lemma \ref{The_polynomial_f} translates to the existence of  $f=aX+b\in \mathbb{Q}_{\leqslant 1}[X]$ such that $\m_{L(\lambda)}(\mu)=f(x)g(x)^{-1} $ for every $x\geqslant \max \{c_j,N\}.$ Consequently  $f$ is an integral multiple of $g$ by  Lemma  \ref{Elementary_result_on_divisibility_between_polynomials} (that is, $f=\m_{L(\lambda)}(\mu)g$),  and thus $\m_{L(\lambda)}(\mu)=ac^{-1},$ where $c,a\in \Q$ are given by \eqref{explicit_value_of_c} and \eqref{explicit_value_of_a}, respectively. The result then immediately follows.

\section{Proof of Proposition \ref{Application_to_type_A}}\label{3rd_pf}     

Let $\mathfrak{g}$ be a simple Lie algebra of type $A_l$ $(l\geqslant 2)$ over $\mathbb{C}.$ For  $  \lambda=\sum_{r=1}^l{a_r\lambda_r} \in \Lambda^+$ a non-zero dominant integral weight,  define $I_{\lambda} =\{r_1,\ldots,r_{N_{\lambda}}\}$ to be maximal in $\{1,\ldots,l\}$ such that $r_1<\ldots <r_{N_{\lambda}}$ and $\prod_{r\in I_{\lambda}}{a_r}\neq 0.$ Also let $\mu=\lambda-(\alpha_1+\cdots+\alpha_l)$ and observe that $\mu$ is $\mathscr{W}$-conjugate to $\lambda-(\alpha_{r_1}+\cdots+\alpha_{r_{N_\lambda}}).$ Therefore 
$$
\m_{L(\lambda)}(\mu)=\m_{L(\lambda)}(\lambda-(\alpha_{r_1}+\cdots+\alpha_{r_{N_\lambda}})),
$$
and in the special case where $N_\lambda=1,$ one immediately deduces that $\m_{L(\lambda)}(\mu)=1.$ Hence  we assume $a_1a_l\neq 0$ in the remainder of this section (so that  $r_1=1,$   $r_{N_{\lambda}}=l,$  and $N_{\lambda}\geqslant 2),$ and aim at showing that the multiplicity of $\mu=\lambda-(\alpha_1+\cdots + \alpha_l)$ in $L(\lambda)$ is given by
\begin{equation}
\m_{L(\lambda)}(\mu)=\prod_{i=2}^{N_{\lambda}}{(r_i-r_{i-1}+1}).
\label{Formula_that_needs_proving}
\end{equation}

We proceed by induction on $N_{\lambda},$ starting by considering the case where $N_{\lambda}=2,$ that is, $I_{\lambda}=\{1,l\}$ and $\lambda=a \lambda_1+b\lambda_l$ for some $a ,b\in \Z_{>0}.$ Even though the result is well-known in this situation, we record an argument here for completeness. Since $\mu=\lambda-(\alpha_1+\cdots + \alpha_l),$ one can apply  Theorem \ref{New_Freudenthal's_formula} to $j=1$ (as $0<1\leqslant a$),  yielding 
$$
\m_{L(\lambda)}(\mu)  = \sum_{\alpha\in \Phi^+_1}{\level_1(\alpha)\m_{L(\lambda)}(\mu+\alpha)} =\sum_{r=1}^{l}{\m_{L(\lambda)}(\mu+\alpha_1+\cdots+\alpha_r)}. 
$$
Now observe that $\mu+\alpha_1+\cdots +\alpha_l=\lambda,$ while $\mu+\alpha_1+\cdots + \alpha_r$ is $\mathscr{W}$-conjugate to $\lambda-\alpha_l$ for every $1\leqslant r <l.$ Therefore $\m_{L(\lambda)}(\mu+\alpha_1+\cdots+\alpha_r)=1$ for every $1\leqslant r\leqslant l$ and hence $\m_{L(\lambda)}(\mu)=l = l-1+1,$ i.e. \eqref{Formula_that_needs_proving} holds. Next assume the existence of $N\in \Z_{>2}$ such  that  \eqref{Formula_that_needs_proving} holds whenever $2\leqslant N_{\lambda}<N,$ and let $\lambda \in \Lambda^+$ be such that $N_{\lambda}=N.$  As in the previous situation, an application of Theorem \ref{New_Freudenthal's_formula} to $j=1$ (as $0<1\leqslant a_1$) yields 
\begin{equation}
\m_{L(\lambda)}(\mu)=\sum_{r=1}^l{\m_{L(\lambda)}(\mu+\alpha_1+\cdots+\alpha_r)}.
\label{[Application]_First_step}
\end{equation}
Notice that for every $1\leqslant r\leqslant r_2-1,$ the weight $\mu+\alpha_1+\cdots + \alpha_r$ is $\mathscr{W}$-conjugate to $\mu+\alpha_1+\cdots + \alpha_{r_2-1},$ so that $\m_{L(\lambda)}(\mu+\alpha_1+\cdots + \alpha_r) = \m_{L(\lambda)}(\mu+\alpha_1+\cdots + \alpha_{r_2-1}).$ Consequently \eqref{[Application]_First_step} becomes
\begin{align}
\m_{L(\lambda)}(\mu)		&= \sum_{r=1}^{r_2-1}{\m_{L(\lambda)}(\mu+\alpha_1+\cdots + \alpha_r)} + \sum_{r=r_2}^{l}{\m_{L(\lambda)}(\mu+\alpha_1+\cdots + \alpha_r)}\cr
						&= (r_2-1)\m_{L(\lambda)}(\mu+\alpha_1+\cdots + \alpha_{r_2-1})+ \sum_{r=r_2}^{l}{\m_{L(\lambda)}(\mu+\alpha_1+\cdots + \alpha_r)}.  
\label{[Application]Second_step}
\end{align}
Now consider the subset $J=\{r_2,r_2+1,\ldots,l\}$ of $\{1,\ldots,l\}$ and let $\mathfrak{g}_J,$ $\mathfrak{h}_J$ be as in Section \ref{Proof_of_the_first_result}. For every $1\leqslant r\leqslant l-r_2+1,$ write $\beta_r=\alpha_{r_2-1+r},$ so that $\Pi_J=\{\beta_1,\ldots,\beta_{l-r_2+1}\}$ forms a base for $\Phi_J,$ and also denote by $\omega_1,\ldots,\omega_{l-r_2+1}$ the corresponding fundamental weights. We then get the restrictions  $\omega=\lambda|_{\mathfrak{h}_J}=\sum_{r=1}^{l-r_2+1}{a_{r_2+r-1}\omega_r}$ and  $\nu=(\mu+\alpha_1+\cdots +\alpha_{r_2-1})|_{\mathfrak{h}_J}=\omega-(\beta_1+\cdots +\beta_{l-r_2+1}).$ Also by Lemma \ref{Restriction_to_suitable_Levi_subalgebra}, we see that \eqref{[Application]Second_step} translates to
$$
\m_{L(\lambda)}(\mu )=(r_2-1)\m_{L(\omega)}(\nu)+\sum_{r=1}^{l-r_2+1}{\m_{L(\omega)}(\nu+\beta_1+\cdots+\beta_r)},
$$
where $L(\omega)$ denotes the irreducible $\mathfrak{g}_J$-module having highest weight $\omega.$ Now a suitable application of Theorem \ref{New_Freudenthal's_formula} shows that the sum on the right-hand side equals $\m_{L(\omega)}(\nu),$ so that 
$$
\m_{L(\lambda)}(\mu)	= r_2\m_{L(\omega)}(\nu).
$$

Finally, observe that $I_{\omega}=\{s_i: 1\leqslant i\leqslant N_{\lambda}-1\}$, where for every $1\leqslant i\leqslant N_{\lambda}-1,$ we have $s_i=r_{i+1}-r_2+1.$
In particular $N_{\omega}=N_{\lambda}-1<N_{\lambda}$ and thus our induction assumption applies, yielding 
$$
\m_{L(\omega)}(\nu)= \prod_{i=2}^{N_{\omega}}{(s_i-s_{i-1}+1)} = \prod_{i=2}^{N_{\lambda}-1}{(r_{i+1}-r_{i}+1)} = \prod_{i=3}^{N_{\lambda}}{(r_i-r_{i-1}+1)},
$$ 
from which the desired result follows.

\section*{Acknowledgements}     

I would like to express my deepest thanks to Professor Donna M. Testerman, for her very helpful comments and suggestions on the earlier versions of this paper.

\bibliographystyle{amsalpha}     

\bibliography{References}  
\end{document}